\documentclass[12pt]{article}
\usepackage{mathrsfs}
\usepackage{amssymb}
\usepackage{amsthm}
\usepackage{graphicx}
\usepackage{authblk}
\usepackage{url}

\usepackage[ruled,lined,linesnumbered,commentsnumbered]{algorithm2e}

\newtheorem{theorem}{Theorem}
\newtheorem{lemma}[theorem]{Lemma}
\newtheorem{corollary}[theorem]{Corollary}
\newtheorem{proposition}[theorem]{Proposition}

\newtheorem{problem}[theorem]{Problem}
\newtheorem{remark}[theorem]{Remark}

\theoremstyle{definition}
\newtheorem{definition}{Definition}

\newcommand{\gtt}{$(\gamma_t - \tau)$}

\begin{document}

\title{A characterization of trees having a minimum vertex cover which is also a minimum total dominating set\thanks{The authors thank the financial support received from
Grant UNAM-PAPIIT IN-114415 and SEP-CONACyT.
Also, the first author would like to thank the support of
the Post-Doctoral Fellowships program of DGAPA-UNAM.}}

\author{C\'esar Hern\'andez-Cruz\thanks{email: cesar@matem.unam.mx (Corresponding Author)} \\ {\small Departamento de Matem\'aticas, Facultad de Ciencias, \\ Universidad Nacional Aut\'onoma de M\'exico, M\'exico} \and Magdalena Lema\'nska\thanks{magda@mif.pg.gda.pl} \\ {\small Department of Technical Physics and Applied Mathematics \\ Gdansk University of
Technology, Poland} \and  Rita Zuazua\thanks{ritazuazua@ciencias.unam.mx} \\ {\small Departamento de Matem\'aticas, Facultad de Ciencias, \\ Universidad Nacional Aut\'onoma de M\'exico, M\'exico}
}

\maketitle
\begin{abstract}
A vertex cover of a graph $G = (V, E)$ is a set
$X \subseteq V$ such that each edge of $G$ is
incident to at least one vertex of $X$. A dominating
set $D \subseteq V$ is a total dominating set of $G$
if the subgraph induced by $D$ has no isolated
vertices. A $(\gamma_t-\tau)$-set of $G$ is a minimum
vertex cover which is also a minimum total dominating
set. In this article we give a constructive characterization
of trees having a $(\gamma_t-\tau)$-set.
\end{abstract}

\section{Introduction}
\label{sec:intro}

Throughout this paper $G=(V,E)$ will be a finite, undirected,
simple and connected graph of order $n$.  The
\emph{neighborhood} of a vertex $v\in V$ is the set $N(v)$
of all vertices adjacent to $v$ in $G$. For a set  $X\subseteq
V,$ the \emph{open neighborhood}, $N(X)$, is defined to be
$\bigcup_{v\in X}N(v)$ and \emph{the closed neighborhood}
of $X$ is defined as $N[X]=N(X)\cup X.$ The degree of a
vertex $v\in V$ is $d(v)=|N(v)|$. A vertex $v\in V$ is an
\emph{end vertex} if $d(v)=1.$ A \emph{support vertex}, or
{\em support}, is the neighbor of an end vertex; a {\em strong
support vertex} is the neighbor of at least two end vertices.
For a set  $S\subseteq V,$ and $v \in S$, the {\em private
neighborhood} $\textnormal{pn}(v,S)$ of $v \in S$ is defined
by $\textnormal{pn}(v,S) = \{ u \in V \colon\ N(u) \cap S =
\{ v \} \}$.   Each vertex in $\textnormal{pn}(v,S)$ is called a
{\em private neighbor} of $v$.

A {\em vertex cover} of  $G$ is a set $X\subseteq V$ such
that each edge of $G$ is incident to at least one vertex of
$X$.  A minimum vertex cover is a vertex cover of smallest
possible cardinality. The \emph{vertex cover number} of
$G$, $\tau(G)$, is the cardinality of a minimum vertex cover
of $G$. A vertex cover of cardinality $\tau(G)$ is called a
$\tau(G)$-set.

The minimum vertex cover problem arises in various
important applications, including multiple sequence
alignments in computational biochemistry (see for
example \cite{pirzadaJKSIAM11}). In computational
biochemistry there are many situations where conflicts
between sequences in a sample can be resolved by
excluding some of the sequences. Of course, exactly
what constitutes a conflict must be precisely defined
in the biochemical context. It is possible to define a
conflict graph where the vertices represent the
sequences in the sample and there is an edge between
two vertices if and only if there is a conflict between the
corresponding sequences. The aim is to remove the
fewest possible sequences that will eliminate all
conflicts, which is equivalent to finding a minimum
vertex cover in the conflict graph $G.$ Several
approaches, such as the use of a parameterized
algorithm \cite{downeyTCS141} and the use of a
simulated annealing algorithm \cite{xuN69},
have been developed to deal with this problem.

A subset $D$ of $V$ is \emph{dominating} in $G$ if
$N[D]=V$. The \emph{domination number} of $G$,
denoted by  $\gamma (G)$, is the minimum cardinality
among all dominating sets in $G.$ A dominating set
$D$ is a \emph{total dominating set} of $G$ if the
subgraph $G[D]$ induced by $D$ has no isolates. In
\cite{cockayneN10}, Cockayne et al. defined the {\em
total domination number} $\gamma_{t}(G)$ of a graph
$G$ to be the minimum cardinality among all total
dominating sets of $G$. A total dominating set of
cardinality $\gamma_{t}(G)$ is called a
$\gamma_{t}(G)$-set.

A total vertex cover is a set which is both a total dominating set
and vertex cover.   In \cite{duttonBICA66}, Dutton studies
total vertex covers of minimum size. He proved that, in
general, the associated decision problem is
$\mathcal{NP}$-complete, and gives some bounds of the
size of a minimum total vertex cover of a graph $G$ in terms of
$\gamma_t (G)$ and $\tau (G)$; this parameter has received
some attention in recent years \cite{duttonDMGT33,li}.   In this
work, we explore a particular case of total vertex covers.   A
\gtt-{\em set} of $G$ is a total vertex cover which is both a
$\gamma_t (G)$-set and a $\tau (G)$-set.   While every graph
has a total vertex cover, by considering $K_2$, it is trivial to
observe that not every graph has a \gtt-set.   So, it is natural
to ask for a characterization of graphs having a \gtt-set.

Clearly, a graph $G$ having a \gtt-set also satisfies $\gamma_t
(G) = \tau (G)$; a graph satisfying this equation will be called
a \gtt-{\em graph}. Again, $K_2$ is an example of a graph which
is not a \gtt-graph, and so, the following question arises:   Does
every \gtt-graph contains a \gtt-set? Unfortuately, the answer is
no (consider the path on $8$ vertices, $P_8$).   So, another
natural problem to consider is to find a characterization of
\gtt-graphs.

Total domination in graphs is well described in
\cite{haynes1998} and recently in \cite{henningDM309} and
\cite{henning2013}. Among the different variants of domination,
total domination is probably the best known and the most widely
studied. Total domination has been successfully related to many
graph theoretic parameters \cite{henning2013}; in particular, an
additional motivation for this work is the following observation.
It is known that for every graph $G$, $\gamma (G) \le \alpha'
(G)$, where $\alpha' (G)$ is the matching number of $G$.
Nonetheless, neither $\alpha' (G)$ nor $\gamma_t (G)$
bounds the other one, and it is an interesting problem to find
families of graphs $G$ such that $\gamma_t (G) \le \alpha' (G)$,
\cite{henning2013}.   On the other hand, in \cite{hartnellCMJ45},
Hartnell and Rall characterized all the graphs $G$ such
that $\gamma (G) = \tau (G)$. Recalling that for every
bipartite graph $G$ we have $\tau (G) = \alpha' (G)$, it is
natural to consider the problem of characterizing bipartite
graphs $G$ such that $\gamma_t (G) = \tau (G)$.   Since
trees are the best-known bipartite graphs, the problem of
characterizing the trees $T$ such that $\gamma_t (G) =
\tau (G)$ seems to be a very good one.

A usual approach in the literature for characterizing
families of trees with a certain property is to consider a
constructive characterization.   First, a family $B$ of
trees having the property $P$ (where it is usually trivial to
verify it) is chosen as a (recursive) base, and then, some
operations preserving $P$ are introduced.   Finally, it is
proved that the family of trees having the property $P$ are
precisely those trees that can be constructed from a tree
in $B$ by recursive applications of the proposed operations.
This approach has been used extensively, to characterize,
for example, Roman trees \cite{henningDMGT22}, trees with
equal independent domination and restrained domination
numbers, trees with equal independent domination and
weak domination numbers \cite{hattinghJGT34}, trees with
equal independent domination and secure domination
numbers \cite{liIPL119}, trees with at least $k$ disjoint
maximum matchings \cite{slaterJCTB25}, trees with equal
$2$-domination and $2$-independence numbers
\cite{brauseDMTCS19}, trees with equal domination and
independent domination numbers, trees with equal
domination and total domination numbers
\cite{dorflingDM306}, etc.   In \cite{dorflingDM306}, a general
framework for studying constructive characterizations of
trees having an equality between two parameters is
discussed.

The main goal of this article is to provide a constructive
characterization of the trees having a \gtt-set.   For
unexplained terms and symbols we refer the reader to
\cite{haynes1998}. The rest of the paper is structured as
follows. In Section \ref{sec:dif} we present some basic
results that will be used in the rest of the paper; it is also
proved that the difference between $\gamma_t (G)$ and
$\tau (G)$ can be arbitrarily large.   Section \ref{sec:trees}
is devoted to prove our main result, we show that the family
of trees $T$ having a \gtt-set can be constructed through
four simple operations starting from $P_4$.   In the final
section some related problems are proposed.


\section{Basic results relating $\gamma_{t}(G)$ and
$\tau (G)$} \label{sec:dif}

In Section \ref{sec:trees} we will define four operations which
will be used to construct all the trees having a \gtt-set.   Such
operations will be defined using the following definition.

\begin{definition}
Let $G=(V(G), E(G))$ and $H=(V(H), E(H))$ be two disjoint
graphs, and let $u$ and $v$ be vertices in $V(G)$ and $V(H)$,
respectively.   The sum of $G$ with $H$ via the edge $uv$,
$G +_{uv} H$, is defined as $V(G +_{uv} H)=V(G) \cup V(H)$
and $E(G +_{uv} H) = E(G) \cup G(H)\cup \{uv\}$.

Moreover, if $H = K_1 = \{ v \}$, we say that we add $v$ to
$G$ supported by $u$.  
\end{definition}

Let $G$ and $H$ be two graphs with $u \in G$ and $v \in H$.
Notice that, regardless of the choice of $u$ and $v$, the
following inequalities are always satisfied: $$\max \{ \gamma_t
(G), \gamma_t (H) \} \le \gamma_t (G +_{uv} H) \le \gamma_t
(G) + \gamma_t (H),$$ $$\max \{ \tau (G), \tau (H) \} \le \tau
(G +_{uv} H) \le \tau (G) + \tau (H).$$   It is also worth noticing
that, for each of the previous four inequalities, there are
examples where they are strict, and examples where they are
equalities; we will come across them in the following sections.

We will now use the previously defined sum to prove that the
difference between $\gamma_t$ and $\tau$ can be arbitrarily
large, even for trees.

\begin{proposition} For any positive integer $k$ there exists a
tree $T_{(k)}$ such that $\tau(T_{(k)}) - \gamma_t(T_{(k)}) = k.$
\end{proposition}

\begin{proof}
Let $P_{4k+2}=(v_1,v_2,\ldots,v_{4k+2})$ be a path. Add
$2k+2$ new vertices to $P_{4k+2}$ supported by the $2k+2$
vertices  $\{ v_1, v_2, v_5, v_6, v_9, v_{10}, \ldots, v_{4k+1},
v_{4k+2} \}$.   The graph that we obtain is a tree $T_{(k)}$ such
that $\gamma_t(T_{(k)})=2k+2$, and $\tau(T_{(k)})=3k+2$.
Thus, we have $\tau(T_{(k)})-\gamma_t(T_{(k)})=k.$   See
Figure \ref{fig1}.
\end{proof}

\begin{figure}[h]
\centering
  \includegraphics[scale=0.75]{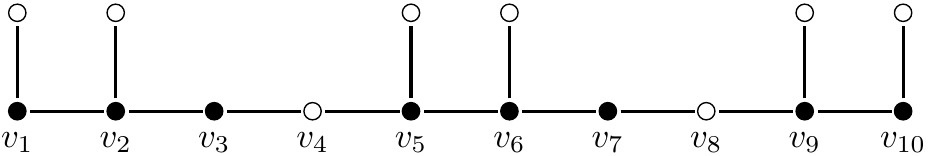}
 \caption{Example of $T_{(k)}$ with  $k=2$.} \label{fig1}
\end{figure}

\begin{proposition} For every positive integer $k$ there exists a
tree $T'_{(k)}$ such that $\gamma_t (T'_{(k)}) - \tau(T'_{(k)})=k.$
\end{proposition}

\begin{proof} 
Let $P_{4k-1}=(v_1,v_2,\ldots,v_{4k-1})$ be a path. Add $2k$
new vertices to $P_{4k-1}$ supported by the vertices with odd
index. The graph that we obtain is a tree $T'_{(k)}$ such that
$\gamma_t(T'_{(k)})=3k, \tau(T'_{(k)})=2k$.   Hence, $\gamma_t
(T'_{(k)})-\tau(T'_{(k)})=k.$   See Figure \ref{fig2}.

\end{proof}

\begin{figure}[h]
\centering
  \includegraphics[scale=0.75]{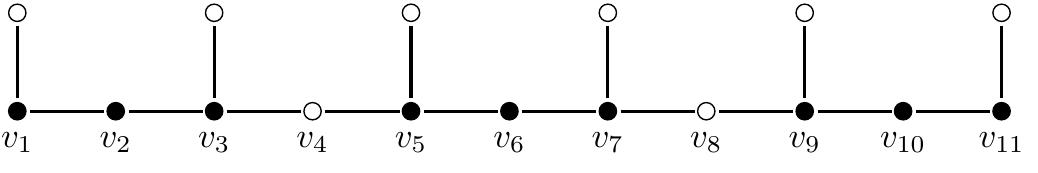}
 \caption{Example of $T'_{(k)}$ with $k=3$.} \label{fig2}
\end{figure}

The following simple remark will be useful in the proof
of our main result.

\begin{remark}  \label{noleafs}
Let $G$ be a graph with at least three vertices.   If $G$ is not a
star, then there exists a minimum total dominating set $D
\subseteq V(G)$ such that $D$ contains no end vertex of $G$.
\end{remark}

\begin{proof}
Let $D$ be a $\gamma_t (G)$-set and $x$ an end vertex of $G$
such that $N(x)=\{y\}$. Then $D-\{x\}\cup \{z\}$ is a total
dominating set  of $G,$ where $z\in N(y)$ is not an end vertex
of $G$.
\end{proof}

Our next result will also be very useful in the following section.

\begin{lemma}\label{lemat1}
If $\gamma_t(G)=\tau(G)$ and $D$ is a \gtt-set of $G$,
then $D$ contains no end vertex of $G.$ 
\end{lemma}

\begin{proof} 
Let $D \subseteq V(G)$ be a ($\gamma_t - \tau$)-set  of $G$.
If $D$ contains an end vertex $x$, then, since $D$ is a total
dominating set, it follows that there exists a vertex $y \in D
\cap N_G (x)$.   This implies that $D \setminus \{ x \}$ is a
vertex cover of $G$, a contradiction to the assumption that
$\gamma_t (G) = \tau (G)$.
\end{proof}

As we mentioned in the introduction, not every tree contains
a \gtt-set.   The smallest tree having a \gtt-set is $P_4$, which
also happens to be the smallest \gtt-tree.   But not every
\gtt-tree contains a \gtt-set.   Actually, it is not hard to find an
infinite class of \gtt-trees not having a \gtt-set, the most simple
one is the family of paths $P_{4k}$, for $k \ge 2$.   Thus, the
class of tress having a \gtt-set is properly contained in the class
of \gtt-trees.

Given a class of graphs, it is common in graph theory to aim for
a characterization in terms of a set of forbidden induced 
subgraphs, because such characterization directly implies
polynomial time recognition for the class. Unfortunately, neither
\gtt-trees, nor trees having a \gtt-set, admit a characterization
of this kind.   To prove this fact, consider the following construction.

Recall that the corona of a graph $G$ is the graph obtained from
$G$ by adding a new vertex $v'$ to $G$ supported by $v$, for
every vertex $v \in V(G)$.   If $H$ is the corona of the graph $G$,
then clearly $V(G)$ is a \gtt-set of $H$.   Hence, any graph
$G$ is an induced subgraph of a \gtt-graph (of a graph having a
\gtt-set), and thus, there exists no forbidden subgraph
characterization of \gtt-graphs (of graphs having a \gtt-set). 

In our next section, we will obtain a constructive characterization
of trees having a \gtt-set.   For this end, we finish this section
introducing a definition and proving a simple technical result.

\begin{definition}
Let $G$ be a graph and $S$ a $\gamma_t (G)$-set.   A vertex
$v$ is $S$-{\em quasi-isolated} if there exists $u \in S$ such
that $\textnormal{pn}(u,S) = \{ v \}$. A vertex $v$ is {\em
quasi-isolated} if it is $S$-quasi-isolated for some $\gamma_t
(G)$-set $S$.
\end{definition}

A vertex $v$ is a $2$-{\em support} if it is at distance two from
an end vertex. The next proposition shows that if a vertex is a
$2$-support, then it is not quasi-isolated.

\begin{proposition}\label{sup-qi}
Let $G$ be a graph and $v\in V$ a $2$-support. Then the
vertex $v$ is non-quasi-isolated.
\end{proposition} 
 
\begin{proof}
Let $x, y, v \in V$ be a leaf, a support and a $2$-support of $G$,
respectively, such that $y \in N(x) \cap N(v)$.   For every
$\gamma_t (G)$-set $S$,  $y \in S$, $v \in N(y)$ and $x \in
\textnormal{pn} (y,S)$, therefore for any $u \in S$,
$\textnormal{pn} (u,S) \ne \{ v \}$.   Hence, $v$ is not
quasi-isolated.
\end{proof}

\section{Trees having a \gtt-set} \label{sec:trees}

As discussed in the previous section, trees having a \gtt-set
do not admit a characterization through a set forbidden
subgraphs.   Following the usual approach in this kind of
situation, we will propose a set of operations preserving the
existence of a \gtt-set to obtain an infinite family of trees having
a \gtt-set, and then, we will prove that every tree having a \gtt-set
belongs to this family.

We define the family $\mathcal{T}$ of trees to consist of all trees
$T$ that can be obtained from a sequence $T_1, T_2, \dots, T_k$
of trees such that $T_1$ is the path $P_4$, $T=T_k$ and, if $k \ge
2$, $T_{i+1}$ can be obtained recursively from $T_i$ by one of
the following operations.

\begin{itemize}
	\item {\bf Operation $\mathcal{O}_1$}:  Consider $u\in V(T)$
		such that $u$ belongs to some \gtt-set. Let $v$ be a
		leaf of a path $P_4$. Then do the sum of $T$ with $P_4$
		via the edge $uv$.

	\item {\bf Operation $\mathcal{O}_2$}:   Let $u\in V(T)$ such
		that $u$ belongs to some \gtt-set. Then add a new vertex
		$v$ to $T$ supported by $u$.  
	
	\item {\bf Operation $\mathcal{O}_3$}:   Let $u\in V(T)$ such
		that $u$ belongs to some \gtt-set and $u$ it is not a
		quasi-isolated vertex. Let $P_2=(v,w)$ be a path with
		two vertices.  Then do the sum of $T$ with $P_2$ via
		the edge $uv$.
		 
	\item {\bf Operation $\mathcal{O}_4$}:   Let $u\in V(T)$ such
		that $u$ is not a quasi-isolated vertex of $T$. Let $v$ be
		a support vertex of a path $P_4$. Then do the sum of $T$
		with $P_4$ via the edge $uv$. 
\end{itemize}

Our next lemma is valid for any tree, not necessarily a tree in
$\mathcal{T}$.

\begin{lemma} \label{O_i}
Let $T$ be a tree.   If $T_i$ is a tree obtained from $T$ by an
operation $\mathcal{O}_i$, $1 \le i \le 4$, then:
\begin{enumerate}
	\item $\gamma_t (T_1) = \gamma_t (T) + 2$
		and $\tau (T_1) = \tau (T) + 2$;
	
	\item $\gamma_t (T_2) = \gamma_t (T)$ and
		$\tau (T_2) = \tau (T)$;

	\item $\gamma_t (T_3) = \gamma_t (T) + 1$
		and $\tau (T_3) = \tau (T) + 1$;
	
	\item $\gamma_t (T_4) = \gamma_t (T) + 2$
		and $\tau (T_4) = \tau (T) + 2$;
\end{enumerate}
and hence,  $\gamma_t (T) - \tau (T) = \gamma_t (T_i) - \tau (T_i)$,
for $1 \le i \le 4$.   In particular $\gamma_t (T) = \tau (T)$ if and only
if $\gamma_t (T_i) = \tau (T_i)$, for $1 \le i \le 4$.   
\end{lemma}

\begin{proof}
Observe that for $1\le i \le 4$, $\gamma_t (T_i) \geq \gamma_t
(T)$ and  $\tau (T_i)\geq \tau (T)$. We consider four cases.

\begin{itemize}

\item Suppose $i = 1$, $P_4=(v,x,y,z)$ and $T_1=T+_{uv}
	P_4.$ Let $S$ be a $\gamma_t(T)$-set (a $\tau (T)$-set,
	respectively). Then,   $S' = S \cup \{ x,y\}$ ($S' = S \cup
	\{ v,y\}$, resp.), is a total dominating set (vertex cover,
	resp.) of $T_1$.   Thus,  $\gamma_t (T_1) \le \gamma_t
	(T) + 2$ and  $\tau (T_1) \le \tau(T) + 2.$  

	For purposes of contradiction, let $D$ be a $\gamma_t
	(T_1)$-set such that $|D| \le \gamma_t (T) + 1.$ Define
	$S=D\cap V(P_4)$, then $2\leq |S|\leq 3.$ Suppose
	$|S|=2$, then $v\notin D$ and $D-S$ is a total dominating
	set of $T$ with cardinality less than or equal to  $\gamma_t
	(T) - 1.$ If $|S| = 3$, then $(D-S)\cup \{w\}$ for $w\in N_T(u)$
	is a total dominating set of $T$ with cardinality less than or
	equal to $\gamma_t (T) - 1.$ Therefore, $\gamma_t (T_1)
	= \gamma_t (T) + 2.$ 

	For purposes of contradiction, let $D$ be a $\tau (T_1)$-set
	such that $|D| \le \tau (T) + 1.$ Define $S=D\cap V(P_4)$,
	then $|S|=2$. Suppose $S = \{ x, y \},$ or $S=\{ x, z\}$ or $S =
	\{ v, y \}$, then $D-S$ is a vertex cover of $T$ with cardinality
	less than or equal to  $\tau (T) - 1.$ Hence, $\tau (T_1) =
	\tau(T) + 2.$     

\item For $i=2$ the proof is straightforward.

\item Suppose $i = 3$, $P_2=(v,w)$ and $T_3=T+_{uv}P_2.$
	Let $S$ be a $\gamma_t(T)$-set such that $u\in S$, then
	$S' = S \cup \{ v\}$ is a total dominating set of $T_3$.
	Similarly, if $S$ is a $\tau (T)$-set then $S' = S \cup \{ v \}$
	is a vertex cover of $T_3$.  Thus,  $\gamma_t (T_3) \le
	\gamma_t (T) + 1$ and  $\tau (T_3) \le \tau(T) + 1.$

	For purposes of contradiction, let $D$ be a $\gamma_t
	(T_3)$-set such that $|D|= \gamma_t (T)$ and there is
	not end vertex in $D$ (such set exists by Remark
	\ref{noleafs}). Then $D\cap V(P_2)=\{v\}$ and $u\in D$.
	Since  $|D-\{v\}|<\gamma_t (T)$, the set $D-\{ v\}$ is not
	a total dominating set of $T$. But, for all $z\in N_T(u)$,
	the set $D' = D - \{ v \} \cup \{ z \}$ is a $\gamma _t (T)$-set
	such that $u$ is $D'$-quasi-isolated, a contradiction. So,
	$\gamma_t (T_3)=\gamma_t (T) + 1$. 

	By definition of vertex cover, it is not posible that $\tau
	(T_3)=\tau(T)$, so $\tau (T_3)=\tau(T) + 1.$

\item Suppose $i=4$, $P_4 = (x, v, y, z)$ and $T_4 = T +_{uv}
	P_4.$ Let $S$ be a $\gamma_t(T)$-set (a $\tau (T)$-set,
	respectively). Then, $S' = S \cup \{ v,y\}$  is a total
	dominating set (vertex cover, resp.) of $T_4$.   Thus,
	$\gamma_t (T_4) \le \gamma_t (T) + 2$ and  $\tau (T_4)
	\le \tau(T) + 2.$  

	For purposes of contradiction, let $D$ be a $\gamma_t
	(T_4)$-set such that $|D| \le \gamma_t (T) + 1.$ Then $D
	\cap V(P_4)=\{v,y\}$. Since $|D-\{ v,y\}|\leq \gamma_t (T) -
	1$, the set $D-\{ v,y\}$ is not a total dominating set of $T.$
	But, for all $w \in N_T(u)$, the set $D' = D - \{ v, y \} \cup \{
	w \}$ is a $\gamma _t (T)$-set such that $u$ is
	$D'$-quasi-isolated, a contradiction. So, $\gamma_t (T_4)
	=\gamma_t (T) + 2$.
 
	By definition of vertex cover, it is not posible that $\tau (T_4)
	\le \tau(T) + 1$, so $\tau (T_3) = \tau(T) + 2.$ 
\end{itemize}
\end{proof}

\begin{corollary} \label{opgttset}
Suppose $T$ is a tree with  $D$ a ($\gamma _t-\tau $)-set of $T$. If
$T_i$ is a tree obtained from $T$ by an operation $\mathcal{O}_i$,
$1 \le i \le 4$,  then $T_i$ has a ($\gamma _t-\tau $)-set $D_i$. 
\end{corollary}

\begin{proof} Let $D$ be a \gtt-set of $T$. With the
notation of the above lemma, we have:

\begin{itemize}
\item If $i=1$ then $D_1=D\cup \{x,y\}$.
\item If $i=2$ then $D_2=D$.
\item If $i=3$ then $D_3=D\cup \{ v\}$.
\item If $i=4$ then $D_4=D\cup \{ v,y\}$.
\end{itemize}
\end{proof}

\begin{theorem}
If $T \in \mathcal{T}$, then $T$ is a ($\gamma_t-\tau$)-tree.  
\end{theorem}

\begin{proof}
Let $T=P_4$, then $\gamma_t (T) = \tau (T)=2$. By Lemma
\ref{O_i} and Corollary \ref{opgttset}, the proof is straightforward.   
\end{proof}

\begin{lemma} \label{quasi-isolated}
Let $T$ be a tree and $u$ a vertex in $T$.
\begin{enumerate}
	\item Let $P_2=(v,w)$ be a path of length two. Suppose
		that $u$ belongs to some $\gamma_t(T)$-set $D$
		of $T$ and define $T'$ to be the sum of $T$ with
		$P_2$ via the edge $uv$.   If $u$ is $D$-quasi-isolated,
		then $\gamma_t (T) = \gamma_t (T')$.
	
	\item Let $v$ and $w$ be the support vertices of a path
		$P_4$.   Define $T'$ to be the sum of $T$ with $P_4$
		via the edge $uv$.   If $u$ is a quasi-isolated vertex,
		then $\gamma_t (T) = \gamma_t (T') + 1$.
\end{enumerate}
\end{lemma}

\begin{proof}
Let $D$ be a $\gamma_t (T)$-set such that $u$ is
$D$-quasi-isolated.   There exists $z \in D$ such that
$\textnormal{pn} (z, D) = \{ u \}$.  It is easy to verify that $D'
= (D \setminus \{ z \}) \cup \{ v \}$, is a $\gamma_t (T')$-set,
in the first case, and $D' = (D \setminus \{ z \}) \cup \{ v, w \}$
is a $\gamma_t (T')$-set for the second case.
\end{proof}

Our main result is the following.

\begin{theorem}
Let $T$ be a tree.   If $T$ has a $(\gamma_t-\tau)$-set,
then  $T \in \mathcal{T}$.
\end{theorem}

\begin{proof} By induction on $n=|V(T)|$. Since
$\gamma_t (T) = \tau (T),$ we have $n\geq 4.$ The
only tree $T$ with four vertices and equality
$\gamma_t (T) = \tau (T)$ is $P_4$, and $P_4\in
\mathcal{T}$. 

Let $T$ be a tree with $n > 4$ and let $D$ be a
($\gamma_t - \tau $)-set of $T$. If $T$ has a strong
support vertex $v$ with a leaf $u$, then $D$ is a
($\gamma_t - \tau $)-set of $T'=T-\{u\}$. By induction
hypothesis $T'\in \mathcal{T}$ and,  using  operation
$\mathcal{O}_2$ we have that $T \in \mathcal{T}$.
Therefore we can assume that there are no strong
support vertices in $T$.

Let $P=(v_0,\ldots, v_l)$ be a longest path in $T$.
Then $d_T(v_1)=2$ and by Lemma \ref{lemat1} the
vertices $v_1, v_2 \in D$. The proof of the theorem
follows to the next two claims.

\vspace{.2cm}

\textbf{Claim 1.} If there exists a vertex $x\in N_T(v_2)
\cap D$ such that $x\neq v_1$ then $d_T(v_2)\neq 2$
and $T \in \mathcal{T}$.

\vspace{.2cm}

Proof of Claim 1. Observe that $d_T (v_2) > 2$.
Otherwise, $d_T (v_2) = 2$ and hence $x=v_3$ and
$D-\{v_2\}$ is a vertex cover of $T$, contradicting
$\gamma_t(T)=\tau(T)$.   If $T' = T-\{v_0, v_1\}$, then
it is not hard to see that $D \setminus \{v_1\}$ is a
\gtt-set of $T'$. From the induction hypothesis $T'\in
\mathcal{T}$. For sake of contradiction, suppose that 
$v_2$ is quasi-isolated in $T'$.   By Lemma
\ref{quasi-isolated},  $\gamma_t (T) = \gamma_t (T')
= \gamma_t (T) - 1$, a contradiction. Therefore, $v_2$
is not quasi-isolated in $T'$, and using operation
$\mathcal{O}_3$, we have that $T\in \mathcal{T}$.

\vspace{.2cm}

\textbf{Claim 2.} If $N_T(v_2)\cap D=\{v_1\}$ then
$d_T(v_3)=2, d_T(v_4)>2$ and $T \in \mathcal{T}$. 

\vspace{.2cm}

Proof of Claim 2. If $N_T(v_2)\cap D=\{v_1\}$, then
$v_2$ is a support vertex or $d_T(v_2)=2.$ 


Observe that if $d_T(v_3)=1$, since $T$ does not
have strong support vertices, then $T=P_4.$ Therefore,
$d_T(v_3)\geq 2$. Since $D$ is a vertex cover of $T$
and $v_3\notin D$,  $|N_T(v_3)\cap D|\geq 2$.


Suppose $v_2$ is a support vertex and let $T'$ be the
tree $T' = T-\{v_0, v_1, v_2,x\},$ where $x$ is the leaf
neighbour of $v_2.$   The set $D'=D-\{v_1,v_2\}$ is a
($\gamma_t - \tau$)-set of $T'$ and, by the induction
hypothesis, $T'\in \mathcal{T}$. Notice that $v_3$ is
not a quasi-isolated vertex of $T'$, otherwise Lemma
\ref{quasi-isolated} would imply $\gamma_t (T) = \gamma_t
(T')+1$, but $\gamma_t (T') = \gamma (T) - 2$. Therefore,
$v_3$ is not a quasi-isolated vertex of $T'$, and we can
obtain $T$ from $T'$ using  operation $\mathcal{O}_4$
and $T\in \mathcal{T}$.



Now we may assume that $d_T (v_2) = 2$.   For purposes
of contradiction, suppose that $d_T(v_3)>2$.   Hence, there
is a path $P_3=(a,b,c)$ which is attached to $v_3$ by the
edge $cv_3$. Since $D$ is a $\gamma_t (T)$-set, we have
$b,c \in D$.     But then $(D \cup \{ v_3 \})  \setminus \{ v_2,
c \}$ is a vertex cover of $T,$ a contradiction. Thus $d_T
(v_3)=2$.


Since $D$ is a vertex cover of $T$ and $v_3 \notin D,$ we
have $v_4\in D$.  If $d_T(v_4)=1$, then $T=P_5$, and $D$
is not a $\gamma _t (T)$-set. If $d_T (v_4) = 2$, then $v_5
\in D$, in this case $(D \setminus \{ v_2, v_4 \}) \cup \{ v_3 \}$
is a vertex cover of $T$, a contradiction.  Hence, $d_T (v_4)
> 2.$ 


Define $T'$ as $T' = T - \{ v_0, v_1, v_2, v_3 \}$, the set 
$D'=D \setminus \{ v_1, v_2 \}$ is a \gtt-set of $T'$
containing $v_4$, and by the induction hypothesis,
$T'\in \mathcal{T}$.   Thus, we can obtain $T$ from
$T'$ using  operation $\mathcal{O}_1$ on $T \in
\mathcal{T}$. 

\end{proof}

Therefore, we have proved the following theorem.

\begin{theorem} \label{main}
It $T$ is a tree, then $T \in \mathcal{T}$ if
and only if $T$ has a \gtt-set. 
\end{theorem}
\vspace{.3cm}


\section{Further work and open problems} \label{sec:gamma-tau}


Once we have characterized the trees having a
\gtt-set, the following natural step is to consider
the following problem.

\begin{problem} \label{prob:gtt}
Find a characterization for the ($\gamma_t - \tau$)-trees.
\end{problem}

If we let $\mathcal{T}'$ be the family of all \gtt-trees,
it is clear that the family $\mathcal{T}$, of all trees
having a \gtt-set, is contained in $\mathcal{T}'$.   We
have already observed in Section \ref{sec:dif}, that
this containment is proper.   Moreover, we can slightly
modify the operations $\mathcal{O}_1, \mathcal{O}_2$,
and $\mathcal{O}_3$ to preserve the equality
$\gamma_t = \tau$, but not necessarily preserving the
existence of a \gtt-set, thus obtaining a larger infinite
family of trees, say $\mathcal{S}$, such that $\mathcal{T}
\subset \mathcal{S} \subset \mathcal{T}'$.   The modified
operations for a tree $T$ are the following (notice the
relaxation of the choice of $u$, cf. Section \ref{sec:dif}).

\begin{itemize}
	\item {\bf Operation $\mathcal{O}'_1$}: Let $u$ be a vertex
		in $T$, and let $v$ be a leaf of a path $P_4$. Then do
		the sum of $T$ with $P_4$ via the edge $uv$.

	\item {\bf Operation $\mathcal{O}'_2$}:   Let $u\in V(T)$ such
		that $u$ belongs to some $\gamma_t (T)$-set and also
		belongs to some $\tau (T)$-set. Then add a new vertex
		$v$ to $T$ supported by $u$.  
	
	\item {\bf Operation $\mathcal{O}'_3$}:   Let $u\in V(T)$ such
		that $u$ belongs to some $\gamma_t (T)$-set and $u$
		it is not a quasi-isolated vertex. Let $P_2=(v,w)$ be a
		path with two vertices.  Then do the sum of $T$ with
		$P_2$ via the edge $uv$.
\end{itemize}

Notice that the family of paths of length $4k$,
$k \ge 2$, mentioned in Section \ref{sec:dif}
as an example of an infinite family of \gtt-graphs
not having a \gtt-set, can be obtained from $P_4$
by recursively applying operation $\mathcal{O}'_1$;
this shows that the inclusion $\mathcal{T} \subset
\mathcal{T}'$ is proper.   Similarly, examples can
be found of a tree $T'$ obtained from a tree $T$
by applying operation $\mathcal{O}_i$, $i \in \{ 2, 3
\}$, such that $T$ has a \gtt-set, but $T'$ does not.

Thus, the family $\mathcal{S}$ above defined is
a good starting point to look for the class of all
\gtt-trees.   It is worth noticing that there are many
ad-hoc operations that could be defined, both on
trees and general graphs, that preserve the
equality $\gamma_t = \tau$ (e.g., subdividing an
edge four times).   Nonetheless, there is no obvious
choice for a set of operations similar to the one
used to prove Theorem \ref{main}, that will lead
to a solution for Problem \ref{prob:gtt}.   Maybe,
instead of a characterization using a set of operations,
the following idea could be useful.   Consider two
\gtt-graphs, $G_i = (V_i, E_i)$, and $X_i \subseteq
V_i$, $i \in \{ 1, 2 \}$, we want to add some edges
joining the vertices of $X_1$ with the vertices of
$X_2$ so that the resulting graph is also a \gtt-graph.
What conditions do we need to achieve this goal?
Consider the following two examples.   First, if
$G_1$ is an empty graph, $|V_1| = |V_2|$, $X_i =
V_i$, for $i \in \{ 1, 2 \}$, and we add a perfect
matching between $X_1$ and $X_2$, we obtain
the corona of the graph $G_2$, which is a \gtt-graph.
Second, if $G_1$ is a $P_4$, $X_1$ is a singleton
containing an end-vertex of $G_1$, and $X_2$ is
a singleton containing any vertex of $G_2$, then
we are describing operation $\mathcal{O}'_1$, and
again, the resulting graph is a \gtt-graph.    These
two ``extremal'' cases, where $X_i$, $i \in \{ 1, 2 \}$
has the largest and smallest possible cardinalities,
respectively, seem to be the easiest to handle.   So,
another kind of recursive characterization could be
obtained if, for example, one could prove that every
\gtt-graph could be obtained by the sum via and
edge, from two smaller \gtt-graphs, or by adding
a perfect matching between two smaller \gtt-graphs.

From the computational point of view, for any tree $T$, both
$\gamma_t (T)$ and $\tau (T)$ can be determined in polynomial
time.   Hence, the problem of determining if $\gamma_t (T) = \tau
(T)$, for a tree $T$, is polynomial time solvable.   For the case of
trees having a \gtt-set, Theorem \ref{main} does not trivially
imply a polynomial algorithm to determine the existence of a
\gtt-set in a tree, so the following problem seems to be interesting.

\begin{problem}
Find the complexity of determining the existence of a
\gtt-set in a tree.
\end{problem}

Of course, it is also interesting to ask both problems for general
graphs.

\begin{problem}
For a given graph $G$:
\begin{itemize}
	\item Find the complexity of determining whether $\gamma_t (G)
		= \tau (G)$.
	\item Find the complexity of determining the existence of a \gtt-set
		in $G$.
\end{itemize}
\end{problem}

Our intuition says that the existence of a \gtt-set is so restrictive in
the structure of $G$ that the second problem might be solved in
polynomial time.

\bibliography{references}{}
\bibliographystyle{plain}
\end{document}